\documentclass[12pt]{article}
\usepackage{amsmath}
\usepackage{amsfonts}
\usepackage{amssymb}

\usepackage{amsthm} 
\usepackage{amscd}
\theoremstyle{plain} 
\newtheorem{theorem}{Theorem}[section]
\newtheorem{corollary}[theorem]{Corollary}

\newtheorem{lemma}[theorem]{Lemma}

\numberwithin{equation}{section}

\newcommand{\al}{\alpha}
\newcommand{\be}{\beta}

\newcommand{\br}{\mathbf {R}}
\newcommand{\bc}{\mathbf {C}}

\newcommand{\om}{\omega}
\newcommand{\ot}{\otimes}
\newcommand{\w}{\wedge}

\newcommand{\p} {\partial}
\newcommand{\g} {\mathfrak {g}}
\newcommand{\h} {\mathfrak {h}}
\newcommand{\sln}{\mathfrak {sl}}
\newcommand{\so}{\mathfrak {so}}
\newcommand{\s}{\mathfrak {sp}}

\begin{document}
\title{A Structure Theorem for Leibniz Homology}
\author{Jerry M. Lodder}
\date{}
\maketitle

\noindent
{\em Mathematical Sciences, Dept. 3MB  \\
Box 30001 \\ 
New Mexico State University \\
Las Cruces NM, 88003, U.S.A. }

\noindent
e-mail:  {\em jlodder@nmsu.edu}

\bigskip
\noindent
{\bf Abstract.}  
Presented is a structure theorem for the Leibniz homology, $HL_*$, 
of an Abelian extension of a simple real Lie algebra $\g$.  As applications,
results are stated for affine extensions of the classical Lie
algebras $\frak{sl}_n(\br)$, $\frak{so}_n(\br)$, and $\frak{sp}_n(\br)$.
Furthermore, $HL_*(\h)$ is calculated when $\h$ is the Lie algebra
of the Poincar\'e group as well as the Lie algebra of the affine 
Lorentz group.  The general theorem identifies all of these in terms 
of $\g$-invariants.



\section{Introduction}

For a (semi-)simple Lie algebra $\g$ over $\br$, the Milnor-Moore
theorem identifies the Lie algebra homology, $H^{\rm Lie}_*(\g ; \,
\br)$, as a graded exterior algebra on the primitive elements of 
$H^{\rm Lie}_*(\g ; \, \br)$, i.e., 
$$ H^{\rm Lie}_*(\g) \simeq \Lambda^*( \rm{Prim}(H_*(\g))), $$
where the coefficients are understood to be in the field $\br$.  
The algebra structure on $H^{\rm Lie}_*(\g)$ can be deduced from the
exterior product of $\g$-invariant cycles on the chain level, and
agrees with the corresponding Pontrajagin product induced from the Lie
group \cite{Koszul}.  For
Leibniz homology, however, we have $HL_n(\g) = 0$, $n \geq 1$, for
$\g$ simple with 
$\br$ coefficients \cite{Ntolo}.  Now, let $\g$ be a simple real Lie
algebra, $\h$ an extension of $\g$ by an Abelian ideal $I$:
$$  \CD 
0 @>>> I @>>> \h @>>> \g @>>> 0
\endCD $$
Then $\g$ acts on $I$, and this extends to a $\g$ action on $\Lambda^*(I)$ by
derivations.  For any $\g$-module $M$, let
$$  M^{\g} = \{ m \in M \ | \ [g, \, m] = 0 \ \ \forall \, g \in \g \}  $$ 
denote the submodule of $\g$-invariants.  Under a mild hypothesis, we
prove that
$$  HL_*(\h) \simeq [\Lambda^*(I)]^{\g} \otimes T(K_*),  $$
where 
$$  K_* = {\rm{Ker}}\big(H^{\rm{Lie}}_*(I; \, \h)^{\g} \to
H^{\rm{Lie}}_{*+1}(\h) \big),  $$ 
and $T(K_*) = \sum_{n \geq 0}K_*^{\otimes n}$ 
is the tensor algebra over $\br$.  Above, $H^{\rm{Lie}}_*(I; \, \h)$
denotes the Lie algebra homology of $I$ with coefficients in $\h$, and
the map $H^{\rm{Lie}}_*(I; \, \h) \to H^{\rm{Lie}}_{*+1}(\h)$ is the
composition
$$ \CD
H^{\rm{Lie}}_*(I; \, \h) @>{j_*}>> H^{\rm{Lie}}_*(\h; \, \h) @>{\pi_*}>> 
H^{\rm{Lie}}_{*+1}(\h),
\endCD $$
where $j_*$ is induced by the inclusion of Lie algebras $j: I
\hookrightarrow \h$, and $\pi_*$ is induced by the projection of chain
complexes
\begin{align*}
& \pi : \h \otimes \h^{\wedge n} \to \h^{\wedge (n+1)} \\
& \pi(h_0 \otimes h_1 \wedge h_1 \wedge \, \ldots \, \wedge h_n) = 
h_0 \wedge h_1 \wedge h_1 \wedge \, \ldots \, \wedge h_n .
\end{align*}
Of course, $K_*$ is computed from the module of $\g$-invariants,
beginning with $H^{\rm{Lie}}_*(I; \, \h)^{\g}$ as indicated above.

The main theorem is easily applied when $\g$ is a classical Lie algebra and
$\h$ is an affine extension of $\g$.  In the final section we state the
results for $\g$ being $\sln_{n} (\br)$, $\so_n(\br)$, $n$ odd or
even, and $I = {\br}^n$, 
whereby $\g$ acts on $I$ via matrix multiplication on vectors, which
is often called the standard representation.  For the (special) orthogonal
Lie algebra, $\frak{so}_n(\br)$, the general theorem agrees with calculations
of Biyogmam \cite{Biyogmam}.  When $\g = \s_n(\br)$
and $I = \br^{2n}$, we recover the author's previous result \cite{Lodder2}.  
Additionally, $HL_*(\h)$ is computed when $\h$ is the Lie algebra of
the Poincar\'e group and the Lie algebra of the affine Lorentz group.

\section{Preliminaries on Lie Algebra Homology}

For any Lie algebra $\g$ over a ring $k$, the Lie algebra homology of
$\g$, written $H^{\text{Lie}}_*(\g ; \, k)$, is the homology of the
chain complex $\Lambda^* (\g )$, namely
$$ \CD 
k @<0<< \g @<[\ , \ ]<< \g^{\wedge 2} @<<< \ldots @<<< \g^{\wedge (n-1)}
@<d<< \g^{\wedge n} @<<< \ldots,  
\endCD $$
where
\begin{align*}
& d(g_1 \wedge g_2 \wedge \, \ldots \, \wedge g_n) = \\
& \sum_{1 \leq i < j \leq n} (-1)^j \, (g_1 \wedge \, \ldots \, \wedge
g_{i-1} \wedge [g_i, \, g_j] \wedge g_{i+1} \wedge \, \ldots \, 
\hat{g}_j \,\ldots \, \wedge g_n). 
\end{align*}
In this paper $H^{\rm{Lie}}_*(\g)$ denotes homology with real
coefficients, where $k = \br$.  Lie algebra homology with
coefficients in the adjoint representation, $H^{\rm{Lie}}_*(\g ; \,
\g)$, is the homology of the chain
complex $\g \otimes \Lambda^*( \g )$, i.e., 
$$ \g \longleftarrow \g \ot \g \longleftarrow \g \ot \g^{\wedge 2} \longleftarrow
\ldots \longleftarrow \g \ot \g^{\wedge (n-1)} \,
{\overset{d}{\longleftarrow}} \,\g \ot \g^{\wedge n} \longleftarrow \ldots ,$$
where
\begin{align*}
& d(g_1 \ot g_2 \wedge g_3 \wedge \, \ldots \, \wedge g_{n+1}) = 
\sum_{i=2}^{n+1} (-1)^i \, ([g_1, \, g_i] \ot g_2 \wedge \, \ldots
  \, \hat{g}_i \, \ldots \, \wedge g_{n+1}) \\
& + \sum_{2 \leq i < j \leq n+1} (-1)^j \, (g_1 \ot g_2 \wedge \, \ldots \, \wedge
g_{i-1} \wedge [g_i, \, g_j] \wedge g_{i+1} \wedge \, \ldots \, 
\hat{g}_j \,\ldots \, \wedge g_{n+1}). 
\end{align*}
The canonical projection $\pi: \g \ot \Lambda^*(\g ) \to \Lambda^{*+1}(\g )$
given by $\g \ot \g^{\wedge n} \to \g^{\wedge (n+1)}$ is a map of
chain complexes, and induces a $k$-linear map on homology
$$ \pi_*: H^{\text{Lie}}_n(\g ; \, \g) \to H^{\text{Lie}}_{n+1}(\g ;\, k). $$
Let $HR_n(\g)$ denote the homology of the complex 
$$  CR_n(\g) = ({\text{Ker}}\, \pi)_n[1] = 
{\text{Ker}}\, [\g \ot \g^{\w (n+1)} \to \g^{\w (n+2)}], \ \ \ n 
\geq 0.  $$
There is a resulting long exact sequence
$$  \CD
\cdots @>{\delta^{\rm{Lie}}}>> HR_{n-1}(\g) @>>> H^{\text{Lie}}_n(\g ; \, \g) @>>>
H^{\text{Lie}}_{n+1}(\g) @>{\delta^{\rm{Lie}}}>>  \\ 
\cdots @>{\delta^{\rm{Lie}}}>> HR_0(\g) @>>> H^{\text{Lie}}_1(\g ; \, \g) @>>>
H^{\text{Lie}}_2(\g) @>{\delta^{\rm{Lie}}}>>  \\ 
@.  0 @>>> H^{\text{Lie}}_0(\g ; \, \g) @>>>
H^{\text{Lie}}_1(\g) @>>> 0.
\endCD  $$

Now let $\g$ be a simple real Lie algebra and $\h$ an
extension of $\g$ by an Abelian ideal $I$.  There is a short exact
sequence of real Lie algebras
$$ \CD 
0 @>>> I @>j>> \h @>{\rho}>> \g @>>> 0,
\endCD $$
where $j: I \to \h$ is an inclusion of Lie algebras, and 
$\rho : \h \to \h /I \simeq \g$ is a projection of Lie algebras.  For
$g \in \g$ and $a \in I$, the action of $\g$ on $I$ can be described
as
$$  [g, \ a] = j^{-1}([h, \ j(a)]),  $$
where $h \in \h$ is any element with $\rho(h) = g$.  The action is
well-defined.  

Conversely, given any representation $I$ of $\g$, such as the standard
representation of a classical real Lie algebra, then $\h$ can be
constructed as the linear span of all elements in $\g$ with all
elements in $I$.  Here $I$ is considered as an Abelian Lie algebra
with $[a, \ b] = 0$ for all $a$, $b \in I$.  Thus, in $\h$, we have
$$  [g_1 + a, \ g_2 + b] = [g_1, \ g_2] + [g_1,\ b] - [g_2,\ a]  $$
for $g_1$, $g_2 \in \g$.

\begin{lemma} \label{2.1}
Let $\g$ be a simple Lie algebra over $\br$, and let
$$ \CD 
0 @>>> I @>j>> \h @>{\rho}>> \g @>>> 0,
\endCD $$
be an Abelian extension of $\g$.  
There are natural vector space isomorphisms
\begin{align}
& H^{\rm{Lie}}_*(\h) \simeq [\Lambda^*(I)]^{\g} \ot H^{\rm{Lie}}_*(\g)\\
& H^{\rm{Lie}}_*(\h ; \, \h) \simeq [H^{\rm{Lie}}_*(I; \, \h)]^{\g} 
\ot H^{\rm{Lie}}_*(\g).
\end{align}
\end{lemma}
\begin{proof}
Apply the homological version of the Hochschild-Serre spectral
sequence to the subalgebra $\g$ of $\h$ \cite{Lodder2}.  Then
$$ H^{\rm{Lie}}_* (\h) \simeq H^{\rm{Lie}}_*(I)^{\g} \ot 
H^{\rm{Lie}}_*(\g).  $$
Since $I$ is Abelian, $H^{\rm{Lie}}_*(I)^{\g} =
[\Lambda^*(I)]^{\g}$, and isomorphism (2.1) follows.  The spectral
sequence yields isomorphism (2.2) directly.  
Note that, since $I$ acts trivially on
$H^{\rm{Lie}}_*(I; \,h)$, we have 
$$ H^{\rm{Lie}}_*(I; \, \h)^{\g} = H^{\rm{Lie}}_*(I; \, \h)^{\h} $$
as well, yielding
$$ H^{\rm{Lie}}_*(\h ; \, \h) \simeq [H^{\rm{Lie}}_*(I; \, \h)]^{\h} 
\ot H^{\rm{Lie}}_*(\g). $$
Compare with Hochschild and Serre \cite{HS}.
\end{proof}

The natural inclusion $\g \hookrightarrow \h$ of Lie algebras leads to
a map of long exact sequences
$$ \CD
@>{\delta^{\rm{Lie}}}>> HR_{n-2}(\g) @>>> H_{n-1}^{\text{Lie}}(\g ; \, \g) @>>>
H_{n}^{\text{Lie}}(\g) @>{\delta^{\rm{Lie}}}>>  \\
@. @VVV  @VVV  @VVV  \\
@>{\delta^{\rm{Lie}}}>>  HR_{n-2}(\h) @>>> H_{n-1}^{\text{Lie}}(\h ; \, \h) @>>>
H_{n}^{\text{Lie}}(\h) @>{\delta^{\rm{Lie}}}>> \, ,
\endCD $$
where ${\delta}^{\rm{Lie}}$ is the connecting homomorphism.  For $\g$
simple, $H^{\rm{Lie}}_{n-1}(\g; \, \g) = 0$, $n \geq 1$ \cite{Hilton},
and 
$$  \delta^{\rm{Lie}}: H^{\rm{Lie}}_n (\g) \to HR_{n-3}(\g)  $$
is an isomorphism for $n \geq 3$.  Note that $H^{\rm{Lie}}_1(\g)
\simeq 0$ and $H^{\rm{Lie}}_2(\g) \simeq 0$.  The inclusion $j : I
\hookrightarrow \h$ is $\g$-equivariant and induces an endormorphism
$$  H^{\rm{Lie}}_*(I; \, \h)^{\g} \overset{j_*}{\longrightarrow} 
H^{\rm{Lie}}_*(\h; \, \h)^{\g} = H^{\rm{Lie}}_*(\h; \, \h).  $$
Recall that every element of $H^{\rm{Lie}}_*(\h; \, \h)$ can in fact be
represented by a $\g$-invariant cycle at the chain level.
Additionally, all elements of $H^{\rm{Lie}}_*(I; \, \h)^{\g}$ can be be
represented by $\g$-invariant cycles, although in general 
$H^{\rm{Lie}}_*(I; \, \h)^{\g}$ is not isomorphic to 
$H^{\rm{Lie}}_*(I; \, \h)$.  
Let $K_*$ be the kernel of the composition
\begin{align*}
& \pi_* \circ j_* : H^{\rm{Lie}}_n(I; \, \h)^{\g}
  \overset{j_*}{\longrightarrow} H^{\rm{Lie}}_n(\h; \, \h)
  \overset{\pi_*}{\longrightarrow} H^{\rm{Lie}}_{n+1}(\h), \\
& K_n = {\rm{Ker}}[H^{\rm{Lie}}_n(I; \, \h)^{\g} \longrightarrow
    H^{\rm{Lie}}_{n+1}(\h)], \ \ \ n \geq 0.
\end{align*}

\begin{theorem}  \label{2.2}
With $I$, $\h$ and $\g$ as in Lemma \eqref{2.1}, we have
$$  HR_n(\h) \simeq \delta^{\rm{Lie}}[H^{\rm{Lie}}_{n+3}(\g)] \oplus
\sum_{i=0}^{n+1}K_{n+1-i}\ot H^{\rm{Lie}}_i(\g), \ n \geq 0.  $$
\end{theorem}
\begin{proof}
The proof follows from the long exact sequence relating $HR_*(\h)$,
$H^{\rm{Lie}}_{*+1}(\h; \, \h)$ and $H^{\rm{Lie}}_{*+2}(\h)$ together
with a specific knowledge of the generators of the latter two homology
groups gleaned from Lemma \eqref{2.1}.  

Note that $H^{\rm{Lie}}_{n}(I; \, \h)^{\g}$ contains
$\Lambda^{n+1}(I)^{\g}$ as a direct summand, induced by a
$\g$-equivariant chain map
\begin{align*}
& \zeta : \Lambda^{n+1}(I) \to \h \ot I^{\w n} \\
& \zeta(a_0 \w a_1 \w \, \ldots \, \w a_n) = \frac{1}{n+1}\sum_{i=0}^n 
(-1)^i a_i \ot a_0 \w a_1 \w \, \ldots \, \hat{a}_i \, \ldots \, \w a_{n},
\end{align*}
where $a_i \in I$.  Then 
$$  \zeta_* : H^{\rm{Lie}}_{n+1}(I)^{\g} = \Lambda^{n+1}(I)^{\g} \to
H^{\rm{Lie}}_{n}(I; \, \h)^{\g}  $$
is an inclusion, since the composition
$$  \pi \circ \zeta : \Lambda^{n+1}(I) \to \h \ot I^{\w n} \to 
\h^{\w (n+1)}  $$
is the identity on $\Lambda^{n+1}(I)$.  Let
$\bar{\Lambda}^*(I)^{\g} = \sum_{k \geq 1}\Lambda^k(I)^{\g}.$
Thus, the morphism
$$  \pi_* : H^{\rm{Lie}}_{*}(\h ; \, \h) \to H^{\rm{Lie}}_{*+1}(\h)  $$
induces a surjection
$$ H^{\rm{Lie}}_{*}(\h ; \, \h) \to \bar{\Lambda}^*(I)^{\g} \ot 
H^{\rm{Lie}}_*(\g) $$
with kernel 
$$ \sum_{n \geq 0} \, \sum_{i=0}^{n+1}K_{n+1-i}\ot
H^{\rm{Lie}}_i(\g).  $$

There is also an inclusion $i_* : H^{\rm{Lie}}_*(\g) \to
H^{\rm{Lie}}_*(\h)$, and it follows that $H^{\rm{Lie}}_{n+3}(\g)$ maps
isomorphically to 
$i_* \circ \delta^{\rm{Lie}} [H^{\rm{Lie}}_{n+3}(\g)]$
in the commutative square
$$ \CD
@>>> H^{\rm{Lie}}_{n+3}(\g) @>{\delta^{\rm{Lie}}}>> HR_n(\g) @>>> \\
@. @V{i_*}VV  @VV{i_*}V  \\
@>>> H^{\rm{Lie}}_{n+3}(\h) @>{\delta^{\rm{Lie}}}>> HR_n(\h) @>>>
\endCD $$
Recall that $H^{\rm{Lie}}_{*}(\g; \, \g) = 0$,  $* \geq 0$, for $\g$ simple.
\end{proof}

\section{Leibniz Homology}

Returning to the general setting of any Lie algebra 
$\g$ over a ring $k$, we recall that the Leibniz homology 
\cite{LP} of $\g$,
written $HL_*(\g ; \, k)$, is the homology of the chain complex
$T(\g)$:
$$  \CD 
k @<0<< \g @<[\ , \ ]<< \g^{\ot 2} @<<< \ldots @<<< \g^{\ot (n-1)}
@<d<< \g^{\ot n} @<<< \ldots,  
\endCD $$
where 
\begin{align*}
& d(g_1, \, g_2, \, \ldots , \,  g_n) = \\
& \sum_{1 \leq i < j \leq n} (-1)^j \, (g_1, \, g_2, \, \ldots, \, 
g_{i-1}, \, [g_i, \, g_j], \, g_{i+1}, \, \ldots \, 
\hat{g}_j \,\ldots , \, g_n), 
\end{align*}
and $(g_1, \, g_2, \, \ldots, \, g_n)$ denotes the element 
$g_1 \ot g_2 \ot \, \ldots \, \ot g_n \in \g^{\ot n}$.  

The canonical projection $\pi' : \g^{\ot n} \to \g^{\w n}$, $n \geq 0$, 
is a map of chain complexes, $T(\g ) \to \Lambda^*(\g )$,  
and induces a $k$-linear map on homology
$$  HL_*(\g ; \, k) \to H^{\text{Lie}}_*(\g ; \, k).  $$
Letting
$$  ({\text{Ker}}\,\pi')_n [2] = {\text{Ker}}\, [\g^{\ot (n+2)} \to
    \g^{\w (n+2)}], \ \ \ n \geq 0,  $$
Pirashvili \cite{Pirashvili} defines the relative theory
$H^{\text{rel}}(\g)$ as the homology of the complex
$$  C^{\text{rel}}_n (\g) = ({\text{Ker}}\,\pi')_n [2],  $$
and studies the resulting long exact sequence relating Lie and Leibniz
homology:  
$$  \CD \label{4.1} 
\cdots @>{\delta}>> H^{\text{rel}}_{n-2}(\g) @>>>   HL_n(\g) @>>>
H^{\text{Lie}}_n(\g) @>{\delta}>>  H^{\text{rel}}_{n-3}(\g) @>>> \\
\cdots @>{\delta}>> H^{\text{rel}}_0(\g) @>>>
HL_2(\g) @>>> H^{\text{Lie}}_2(\g) @>>> 0 \\
@.  0 @>>> HL_1(\g) @>>> H^{\text{Lie}}_1(\g) @>>> 0 \\
@.  0 @>>> HL_0(\g) @>>> H^{\text{Lie}}_0(\g) @>>> \phantom{.}0. 
\endCD $$

The projection $\pi' : \g^{\ot (n+1)} \to \g^{\w (n+1)}$ can be
factored as the composition of projections
$$  \g^{\ot (n+1)} \longrightarrow \g \ot \g^{\w n}  
\longrightarrow \g^{\w (n+1)},  $$
which leads to a natural map between exact sequences
$$  \CD
H^{\rm{rel}}_{n-1}(\g)  @>>> HL_{n+1}(\g) @>>>
H^{\rm{Lie}}_{n+1}(\g)  @>{\delta}>> H^{\rm{rel}}_{n-2}(\g) \\
@VVV   @VVV   @V{\mathbf{1}}VV   @VVV  \\
HR_{n-1}(\g)  @>>> H^{\rm{Lie}}_n(\g ; \, \g)  @>>>
H^{\rm{Lie}}_{n+1}(\g)  @>{\delta^{\rm{Lie}}}>> HR_{n-2}(\g) 
\endCD  $$

A key technique in the calculation of Leibniz homology is the
Pirashvili spectral sequence \cite{Pirashvili}, which converges to the
relative groups $H^{\rm{rel}}_*$.  Consider the filtration of
$$ C^{\rm{rel}}_n(\g) = {\rm{Ker}}(\g^{\ot (n+2)} \to \g^{\w (n+2)}), 
\ \ \ n \geq 0,  $$
given by 
$$ \mathcal{F}_m^k(\g) = \g^{\ot k} \ot {\rm{Ker}}( \g^{\ot(m+2)} \to
\g^{\w (m+2)}), \ \ \ m \geq 0, \ k \geq 0. $$
Then $\mathcal{F}_{m-1}^*$ is a subcomplex of $\mathcal{F}_m^*$, and
\begin{align*}
E_{m, \, k}^0  & = {\mathcal{F}}^k_m / {\mathcal{F}}^{k+1}_{m-1} \\
& \simeq \g^k \ot {\rm{Ker}}(\g \ot \g^{\w (m+1)} \to \g^{\w (m+2)}) \\
& = \g^k \ot CR_m(\g).
\end{align*}
From \cite{Pirashvili}, we have
$$  E_{m, \, k}^2 \simeq HL_k(\g) \ot HR_m(\g), \ \ \ m \geq 0, \ k
\geq 0.  $$

\begin{lemma}  \label{3.1}
Let $0 \to I \to \h \to \g \to 0$ be an Abelian
  extension of a simple real Lie algebra $\g$.  Then there is a
  natural injection
$$  \epsilon_* : H^{\rm{Lie}}_{*}(I; \, \h)^{\g} \to HL_{*+1}(\h)  $$ 
induced by a $\g$-equivariant chain map
$$ \epsilon_n : \h \ot \Lambda^n(I) \to \h^{\ot (n+1)}, \ n \geq 0.  $$
\end{lemma}
\begin{proof}
For $b \in \h$, $a_i \in I$, $i = 1, \ 2, \, \ldots \, n$, define
$$ \epsilon_n(b \ot a_1 \w a_2 \w \, \ldots \, \w a_n) =
\frac{1}{n!} \sum_{\sigma \in S_n}{\rm{sgn}}(\sigma) \, 
b \ot a_{\sigma(1)} \ot a_{\sigma(2)} \ot \, \ldots \, \ot
a_{\sigma(n)}.  $$
Since $[a_i , \ a_j] = 0$ for $a_i$, $a_j \in I$, it follows that
$$ d_{\rm{Lieb}} \circ \epsilon_n = \epsilon_{n-1} \circ
d_{\rm{Lie}}.  $$
Also, $\epsilon_n$ is $\g$-equivariant, since $\g$ acts by derivations
on both $\h \ot \Lambda^n(I)$ and $\h^{\ot (n+1)}$.  Thus, there is an
induced map
$$  \epsilon_* : H^{\rm{Lie}}_*(I ; \,\h)^{\g} \to HL_{*+1}(\h)^{\g} =
HL_{*+1}(\h).  $$
The composition
$$  \pi' \circ \epsilon_n : \h \ot \Lambda^n(I) \to \h^{\ot (n+1)} \to
\h \ot \h^{\w n}  $$
is the identity on $\h \ot \Lambda^n(I)$.  Since $H^{\rm{Lie}}_*(\h ; \,
\h)$ contains $H^{\rm{Lie}}_*(I; \, \h)^{\g}$ as a direct summand via
$H^{\rm{Lie}}_*(I; \, \h)^{\g} \ot H^{\rm{Lie}}_0(\g)$ (see Lemma
\eqref{2.1}), it follows that $(\pi' \circ \epsilon)_*$ and $\epsilon_*$
are injective.
\end{proof}
\begin{lemma} \label{3.2}
With $I$, $\h$, $\g$ as in Lemma \eqref{3.1}, there is a vector space 
splitting (that is a splitting of trivial $\g$-modules)
$$  H^{\rm{Lie}}_n(I; \, \h)^{\g} \simeq [\Lambda^{n+1}(I)]^{\g} \oplus
K_n,  $$
where
$$  K_n = {\rm{Ker}}[H^{\rm{Lie}}_n(I; \, \h)^{\g} \to
  H^{\rm{Lie}}_{n+1}(\h)].  $$
\end{lemma}
\begin{proof}
The proof begins with the $\g$-equivariant chain map
$$  \zeta : \Lambda^{n+1}(I) \to \h \ot I^{\w n} $$
constructed in Theorem \eqref{2.2}.  
Recall that $K_n$ is
defined as the kernel of the composition
$$ H^{\rm{Lie}}_n(I; \, \h)^{\g} \to H^{\rm{Lie}}_n(\h ; \, \h) \to
H^{\rm{Lie}}_{n+1}(\h).  $$
Note that $H^{\rm{Lie}}_n(\h ; \, \h)$ contains $H^{\rm{Lie}}_n(I; \,
\h)^{\g}$ as a summand from Lemma \eqref{2.1}.    
\end{proof}
 
To begin the calculation of the differentials in the Pirashvili
spectral sequence converging to $H^{\rm{rel}}(\h)$, first consider the 
spectral sequence converging to $H^{\rm{rel}}(\g)$, where $\g$ is simple.
We have $H^{\rm{Lie}}_n(\g ; \, \g) = 0$ for $n \geq 0$
from \cite{Hilton} and $HL_n(\g) = 0$ for $n \geq 1$ from
\cite{Ntolo}.  It follows that $\delta^{\rm{Lie}}:
H^{\rm{Lie}}_{n+3}(\g) \to HR_n(\g)$ and $\delta : H^{\rm{Lie}}_{n+3}(\g)
\to H^{\rm{rel}}_n(\g)$, $n \geq 0$, are isomorphisms in the square
$$ \CD
@>>> H^{\rm{Lie}}_{n+3}(\g) @>{\delta}>> H^{\rm{rel}}_n(\g) @>>> \\
@. @V{\bf{1}}VV  @VV{\bf{1}}V  \\
@>>> H^{\rm{Lie}}_{n+3}(\g) @>{\delta^{\rm{Lie}}}>> HR_n(\g) @>>> 
\endCD $$
\begin{lemma} In the Pirashvili spectral sequence converging to
  $H^{\rm{rel}}_*(\g)$ for a simple real Lie algebra $\g$, all higher
  differentials 
$$ d^r :E^r_{m, \, k} \to E^r_{m-r, \, k+r-1} , \ \ \ r \geq 2, $$
are zero.
\end{lemma}
\begin{proof}
Since $E^2_{m, \, k} \simeq HL_k(\g) \ot HR_m (\g)$, and $HL_k(\g) = 0$
for $k \geq 1$, it is enough to consider
$$  E^2_{m, \, 0} \simeq \br \ot HR_m(\g) \simeq
\delta^{\rm{Lie}}[H^{\rm{Lie}}_{m+3}(\g)],  $$
for which $d^r[E^r_{m, \, 0}] = 0$, $r \geq 2$.
\end{proof}

By naturality of the Pirashvili spectral sequence, 
$$  \mathcal{F}^*_m(\g) \hookrightarrow \mathcal{F}^*_m(\h).  $$
Thus, in the spectral sequence converging to $H^{\rm{rel}}_*(\h)$, 
we also have
$$  d^r[\delta^{\rm{Lie}}[H^{\rm{Lie}}_{m+3}(\g)] = 0, \ \ \ r \geq
  2.  $$

Due to the recursive nature of $H^{\rm{rel}}_*(\h)$ with 
$E^2_{m, \, k} \simeq HL_k(\h) \ot HR_m(\h)$, we calculate $HL_n(\h)$
by induction on $n$.  We have immediately
\begin{align*}
&HL_0(\h) \simeq \br \\
&HL_1(\h) \simeq H^{\rm{Lie}}_1(\h) \simeq I^{\g} \simeq H_0(I; \,
  \h)^{\g} \\
&HL_2(\h) \simeq H^{\rm{Lie}}_1(\h ; \, \h) \simeq H^{\rm{Lie}}_1(I;
  \, \h)^{\g}.
\end{align*}
Elements in 
$$  HL_1(\h) \ot HR_0(\h) + HL_0(\h) \ot HR_1(h)  $$
determine $H^{\rm{rel}}_1(\h)$, which then maps to $HL_3(\h)$, etc.
We now construct the differentials in the Pirashvili spectral 
sequence.

\begin{lemma}  \label{PSS-1}
In the Pirashvili spectral sequence converging to
$H^{\rm{rel}}_*(\h)$, the differential 
$$ d^r[E^2_{m, \, 0}] = d^r[HL_0(\h) \otimes HR_m(\h)] = d^r[{\br}
  \ot HR_m(\h)]  $$ 
is given by
\begin{align*}
& d^r[\delta^{\rm{Lie}}(H^{\rm{Lie}}_{m+3}(\g))] = 0, \ \ \ r \geq 2,\\
& d^r[K_{m+1}] = 0, \ \ \ r \geq 2, \\
& d^{m+3-p}: K_{m+1-p} \ot H^{\rm{Lie}}_p(\g) \to K_{m+1-p} \ot 
\delta^{\rm{Lie}}[H^{\rm{Lie}}_p(\g)], \ {\rm{where}} \\
& K_{m+1-p} \ot \delta^{\rm{Lie}}[H^{\rm{Lie}}_p(\g)] \subseteq
HL_{m+2-p}(\h) \ot HR_{p-3}(\h) \subseteq E^2_{p-3, \, m+2-p}, \\
& d^{m+3-p}(\om_1 \ot \om_2) = \om_1 \ot \delta(\om_2).  
\end{align*}
\end{lemma}
\begin{proof}
Since elements in $K_{m+1}$ are represented by cycles in
$H^{\rm{Lie}}_{m+1}(\h; \, I)^{\g}$ that map via $\epsilon_*$ to cycles
in $HL_{m+2}(\h)$, it follows $d^r[K_{m+1}] =  0$, $r \geq 2$.  

Let $[\om_2] \in H^{\rm{Lie}}_p(\g)$.  Consider  
$$  \om_2 = \sum_{i_1, i_2, \ldots , i_p}g_{i_1} \w g_{i_2} \w \,
\ldots \, \w g_{i_p} \in \g^{\w p}. $$
Now, using the homological algebra of the long exact sequence relating
Leibniz and Lie-algebra homology, 
$$ \tilde{\om}_2 = \sum_{i_1, i_2, \ldots , i_p}g_{i_1} \ot g_{i_2} \ot \,
\ldots \, \ot g_{i_p} \in \g^{\ot p}  $$
is a chain with $\pi' (\tilde{\om}_2) = \om_2$ and 
$d(\tilde{\om}_2) \neq 0$ in the Leibniz complex.  Moreover,
$$ d(\tilde{\om}_2) \in {\rm{Ker}}
\big( \g^{\ot (p-1)} \to \g^{\w (p-1)} \big) $$
and $d(\tilde{\om}_2)$ represents the class $\delta(\om_2)$ in
$H^{\rm{rel}}_{p-3}(\g) \simeq HR_{p-3}(\g)$.  Since $w_1 \in
K_{m+1-p}$ is $\g$-invariant, $[\om_1, \ g] = 0$ $\forall \, g \in
\g$.  Also, $\epsilon (\om_1)$ is a $\g$-invariant cycle representing
$\om_1$ in $T(\g)$.  Thus,
$$  d^{m+3-p}(\om_1 \ot \om_2) = \om_1 \ot d(\tilde{\om}_2) = \om_1 \ot
\delta(\om_2).  $$
\end{proof}
 
By a similar argument to that in Lemma \eqref{PSS-1}, we have that
$$ d^r[\Lambda^*(I)^{\g} \ot \delta^{\rm{Lie}}(H^{\rm{Lie}}_{*+3}(\g))] = 0,
\ \ \  r \geq 2, $$
where 
$$ \Lambda^k(I)^{\g} \ot \delta^{\rm{Lie}}(H^{\rm{Lie}}_{m+3}(\g))
\subseteq HL_k(\h) \ot HR_m(\h) \subseteq E^2_{m, \, k}. $$
Thus, the elements in 
$M = \Lambda^*(I)^{\g} \ot \delta^{\rm{Lie}}(H^{\rm{Lie}}_{*+3}(\g))$
represent absolute cycles in $H^{\rm{rel}}_*(\h)$.   We claim that
$$  M \subseteq {\rm{Ker}}(H^{\rm{rel}}_*(\h) \to HL_{*+2}(\h)),  $$
which follows from:
\begin{lemma}
The boundary map
$$ \delta: H^{\rm{Lie}}_{*+3}(\h) \to H^{\rm{rel}}_*(\h) $$
satisfies
$$ \delta[\Lambda^k(I)^{\g} \ot H^{\rm{Lie}}_{m+3}(\g)] =
\Lambda^k(I)^{\g} \ot \delta^{\rm{Lie}}[H^{\rm{Lie}}_{m+3}(\g)], $$
for $k \geq 0$, $m \geq 0$. 
\end{lemma}
\begin{proof}
The proof follows from the long exact sequence relating Lie and
Leibniz homology by choosing representatives for the Lie classes
$\Lambda^k(I)^{\g} \ot H^{\rm{Lie}}_{m+3}(\g)$ at the chain level.
Also, note that for $\g$ simple, $H^{\rm{Lie}}_1(\g) = 0$ and 
$H^{\rm{Lie}}_2(\g) = 0$.  
\end{proof}

Now let
$$  \om_1 \ot (\om_2 \ot \om_3) \in \Lambda^k(I)^{\g} \ot (K_{m+1-p}
\ot H^{\rm{Lie}}_p(\g)) \subseteq E^2_{m, \, k}.  $$
Then, using the $\g$-invariance of elements in $\Lambda^*(I)^{\g}$ and
$K_*$, as well as the $\h$-invariance of $\Lambda^*(I)^{\g}$, we have 
\begin{align*}
d^{m+3-p}&(\om_1 \ot \om_2 \ot \om_3) =  
(\om_1 \ot \om_2 )\ot \delta(\om_3) \in \\
& (\Lambda^k(I)^{\g} \ot K_{m+1-p}) \ot
\delta^{\rm{Lie}}[H^{\rm{Lie}}_p(\g)] \\ 
& \subseteq HL_{k+m+2-p}(\h) \ot HR_{p-3}(\g) \subseteq E^2_{p-3, \,
  k+m+2-p}.  
\end{align*}

For the remainder of this section, we suppose: 

\bigskip
\noindent
{\bf Hypothesis A}:  {\em Every element of $K_n$ has an $\h$-invariant 
representative in $HL_{n+1}(\h)$ at the chain level.}

\bigskip
Since $I$ acts trivially on $H^{\rm{Lie}}_*(I; \, \h)$, we have
$H^{\rm{Lie}}_*(I; \, \h)^{\g} = H^{\rm{Lie}}_*(I; \, \h)^{\h}$, 
and Hypothesis A is reasonable.  The end of this section offers a 
canonical construction of $\h$-invariants.

\begin{theorem} \label{main-theorem} 
Let $0 \to I \to \h \to \g \to 0$ be an Abelian extension of a simple
real Lie algebra $\g$.  Then, under Hypothesis A, we have
$$  HL_*(\h) \simeq \Lambda^*(I)^{\g} \ot T(K_*),  $$
where $T(K_*) = \sum_{n \geq 0}K_*^{\ot n}$ denotes the tensor algebra,
and
$$ K_n = {\rm{Ker}}[H^{\rm{Lie}}_n(I; \, \h)^{\g} \to
  H^{\rm{Lie}}_{n+1}(\h)].  $$
\end{theorem}
\begin{proof}
It follows by induction on $\ell$ that for certain $r$
$$  d^r[K_*^{\ot \ell} \ot H_*^{\rm{Lie}}(\g)] \to
K_*^{\ot \ell} \ot \delta^{\rm{Lie}}[H_*^{\rm{Lie}}(\g)]  $$
is given by
$$  d^r(\om_1 \ot \om_2 \ot \, \ldots \, \ot \om_{\ell} \ot v) =
\om_1 \ot \om_2 \ot \, \ldots \, \ot \om_{\ell} \ot \delta(v), $$
where $\om_i \in K_*$ and $v \in H_*^{\rm{Lie}}(\g)$.  By a similar
induction argument, we have
$$ d^r[\Lambda^*(I)^{\g} \ot K_*^{\ot \ell} \ot H_*^{\rm{Lie}}(\g)] 
\to \Lambda^*(I)^{\g} \ot K_*^{\ot \ell} \ot 
\delta^{\rm{Lie}}[H_*^{\rm{Lie}}(\g)]  $$ 
is given by
$$  d^r(u \ot \om_1 \ot \om_2 \ot \, \ldots \, \ot \om_{\ell} \ot v) =
u \ot \om_1 \ot \, \ldots \, \ot \om_{\ell} \ot \delta(v), $$
where $u \in \Lambda^*(I)^{\g}$, $\om_i \in K_*$ and 
$v \in H_*^{\rm{Lie}}(\g)$.  The only absolute cycles in the 
Pirashvili spectral sequence are elements of
$$  \Lambda^*(I)^{\g} \ot K_*^{\ot \ell},  $$
which are not in 
${\rm{Im}} \, \delta: H_{*+3}^{\rm{Lie}}(\h) \to H^{\rm{rel}}_*(\h)$.
By induction on $\ell$, 
$$  HL_*(\h) \simeq \Lambda^*(I)^{\g} \ot T(K_*).  $$
\end{proof}

We now study elements in $K_n$ and outline certain canonical 
constructions to produce $\h$-invariants.  Recall that 
$H_*^{\rm{Lie}}(I; \, \h)^{\g}$ is the homology of
$$  \CD 
\h^{\g} @<[\ , \ ]<< (\h \ot I)^{\g} @<d<< (\h \ot I^{\wedge 2})^{\g} @<d<< 
\ldots @<d<< (\h \ot I^{\wedge n})^{\g} @<d<<  . 
\endCD $$
Note that as $\g$-modules, $\h \simeq \g \oplus I$ and
$$ \h \ot I^{\wedge n} \simeq (\g \ot I^{\wedge n}) \oplus 
(I \ot I^{\wedge n}). $$
Thus, 
$(\h \ot I^{\wedge n})^{\g} \simeq (\g \ot I^{\wedge n})^{\g} 
\oplus (I \ot I^{\wedge n})^{\g}$.  Any element in $K_n$ having a 
representative in $(I \ot I^{\wedge n})^{\g}$ is necessarily an $\h$-invariant
at the chain level, since $I$ acts trivially on 
$(I \ot I^{\wedge n})^{\g}$.  Of course, all elements of 
$(I \ot I^{\wedge n})^{\g}$ are cycles.  

\begin{lemma} \label{alpha-inv}
Any non-zero element in $K_n$ having a representative in 
$$(\g \ot I^{\wedge n})^{\g}$$ is determined by an injective map of
$\g$-modules, $\al : \g \to I^{\wedge n},$
where $\g$ acts on itself via the adjoint action.
\end{lemma}
\begin{proof}
Let $B : \g \, \overset{\simeq}{\longrightarrow} \, \g^* = 
{\rm{Hom}}_{\br}(\g, \ \br)$ be the isomorphism from a simple Lie
algebra to its dual induced by the Killing form.  Then the composition
$$  \g \ot I^{\wedge n} \, \overset{B \ot \bf{1}}{\longrightarrow} \,
\g^* \ot I^{\wedge n} \, \overset{\simeq}{\longrightarrow} \,
{\rm{Hom}}_{\br}(\g, \ I^{\wedge n})  $$
is $\g$-equivariant, and induces an isomorphism
$$  (\g \ot I^{\wedge n})^{\g} \, \overset{\simeq}{\longrightarrow} \,
{\rm{Hom}}_{\g}(\g, \ I^{\wedge n}).  $$
Since $\g$ is simple, a non-zero map of $\g$-modules
$\al : \g \to I^{\wedge n}$ has no kernel, and $\g \simeq {\rm{Im}}\, (\al)$.  
\end{proof} 

Consider the special case where $I \simeq \g$ as $\g$-modules, although
$I$ remains an Abelian Lie algebra.  Let $\al : \g \to I$ be a 
$\g$-module isomorphism and let $B^{-1} : \g^* \to \g$ be the inverse
of $B : \g \to \g^*$ in the proof of Lemma \eqref{alpha-inv}.  For 
a vector space basis $\{ b_i \}_{i=1}^n$ of $\g$, let $\{ b_i^* \}_{i=1}^n$ 
denote the dual basis.
\begin{lemma}
With $\al : \g \simeq I$ as above, the balanced tensor 
$$ \om = \sum_{i=1}^n B^{-1}(b_i^*) \ot \al(b_i) + \al(B^{-1}(b_i^*)) \ot b_i
\in \h \ot \h  $$
is $\h$-invariant.
\end{lemma}
\begin{proof}
By construction,
$$  \sum_{i=1}^n B^{-1}(b_i^*) \ot \al(b_i) \in \g \ot I  
\hookrightarrow \h \ot \h  $$
is $\g$-invariant.  Since $\al$ is an isomorphism of $\g$-modules, it
follows that 
$$  \sum_{i=1}^n  \al(B^{-1}(b_i^*)) \ot b_i \in I \ot \g  
\hookrightarrow \h \ot \h  $$
is also a $\g$-invariant.  Now, let $a \in I$.  There is some $g_0 \in \g$
with $\al(g_0) = a$.  Thus, 
\begin{align*}
& [\om, \ a] = [\om, \ \al(g_0) ] \\
& = \sum_{i=1}^n [B^{-1}(b_i^*) \ot \al(b_i), \ \al(g_0)] + 
[\al(B^{-1}(b_i^*)) \ot b_i, \ \al(g_0)]  \\
& = \sum_{i=1}^n \al\big([B^{-1}(b_i^*),\ g_0]\big) \ot \al(b_i) + 
\al(B^{-1}(b_i^*)) \ot \al\big( [b_i, \ g_0] \big)  \\
& = \sum_{i=1}^n (\al \ot \al) \big( [B^{-1}(b_i^*) \ot b_i, \ g_0] \big)
= 0.
\end{align*} 
Since $\al : \g \to I$ is an isomorphism of $\g$-modules, it follows 
that if 
$$ \sum_{i=1}^n B(b_i^*) \ot \al(b_i) \in \g \ot I  $$ 
is a $\g$-invariant, then
$\sum_{i=1}^n B(b_i^*) \ot b_i \in \g \ot \g$ is also a $\g$-invariant.
\end{proof}

\section{Applications}

We compute the Leibniz homology for extensions of the classical Lie
algebras $\frak{sl}_n(\br)$, $\frak{so}_n(\br)$, and $\frak{sp}_n(\br)$.
Additionally, $HL_*$ is calculated for the Lie algebra of the 
Poincar\'e group $\br^4 \rtimes   SL_2( \bc )$ and the 
Lie algebra of the affine Lorentz group $\br^4 \rtimes SO(3,1)$.  
To describe a common setting
for these examples, let $\g$ be a (semi-)simple real Lie algebra, 
and consider $\g \subseteq \frak{gl}_n(\br)$.  Then $\g$ acts on
$I = \br^n$ via matrix multiplication on vectors in $\br^n$, which
is often called the standard representation.  Consider
$$  \frac{\p}{\p x^1}, \ \frac{\p}{\p x^2}, \ \ldots, \ 
\frac{\p}{\p x^n}  $$ 
as a vector space basis for $\br^n$.  Then the elementary matrix with 1 in row 
$i$, column $j$, and 0s everywhere else becomes $x_i \frac{\p}{\p x^j}$.  
In the sequel, $\h$ denotes the real Lie algebra formed via the
extension
$$  \CD 
0 @>>> I @>>> \h @>>> \g @>>> 0 \, .  
\endCD $$
Also, the element
$$ \frac{\p}{\p x^1} \w \frac{\p}{\p x^2} \w \, \ldots \, \w 
\frac{\p}{\p x^n} \in I^{\w n}  $$
is the volume form, and often occurs as a $\g$-invariant.

\begin{corollary} \label{reductive}
Let $\g$ be a simple Lie algebra and I an Abelian Lie algebra,
both over $\br$.  If $\h \simeq \g \oplus I$ as Lie algebras,
i.e., $\h$ is reductive, then
\begin{align*}
& HL_*(\h) \simeq \Lambda^*(I) \ot T(K_*), \\
& K_* = {\rm{Ker}}(I \ot \Lambda^*(I) \to \Lambda^{*+1}(I))
\end{align*}
\end{corollary}
\begin{proof}
Since $\g$ acts trivially on $I$, it follow that
\begin{align*}
& [\Lambda^*(I)]^{\g} = \Lambda^*(I), \\
& H^{\rm{Lie}}_*(\h) \simeq H^{\rm{Lie}}_*(\g) \ot \Lambda^*(I), \\
& H^{\rm{Lie}}_*(I; \, \h)^{\g} \simeq I \ot \Lambda^*(I).  
\end{align*}
Thus, $K_n = {\rm{Ker}}(I \ot \Lambda^n(I) \to \Lambda^{n+1}(I))$.
\end{proof}
By way of comparison, from \cite{Loday1996}, under the hypotheses of 
Corollary \eqref{reductive}, we have
$$ HL_*(\h) \simeq HL_*(\g) * HL_*(I) \simeq HL_*(I) \simeq T(I).  $$
Thus, as vector spaces, $\Lambda^*(I) \ot T(K_*) \simeq T(I)$.  For
tensors of degree two, the above isomorphism becomes
$$  \Lambda^2(I) \oplus S^2(I) \simeq I^{\ot 2},  $$
where $S^2(I)$ denotes the second symmetric power of $I$.  

\begin{corollary}
For $\g = \frak{sl}_n(\br)$ and $I = \br^n$ the standard representation
of $\frak{sl}_n(\br)$, we have
$$  HL_*(\h) \simeq [\Lambda^*(I)]^{\frak{sl}_n(\br)} =
\big\langle \frac{\p}{\p x^1} \w \frac{\p}{\p x^2} \w \, \ldots \, \w 
\frac{\p}{\p x^n} \big\rangle .  $$
\end{corollary}

\begin{proof} In this case there are no non-trivial 
$\frak{sl}_n(\br)$-module maps from $\frak{sl}_n(\br)$ to $I^{\w k}$.
Using Lemma \eqref{alpha-inv}, we have
\begin{align*}
& H^{\rm{Lie}}_*(I; \, \h)^{\frak{sl}_n(\br)} \simeq 
H^{\rm{Lie}}_*(I; \, I)^{\frak{sl}_n(\br)} \simeq
[I \ot \Lambda^*(I)]^{\frak{sl}_n(\br)} \\ 
& = \big\langle \frac{\p}{\p x^1} \w \frac{\p}{\p x^2} \w \, \ldots \, \w 
\frac{\p}{\p x^n} \big\rangle .  
\end{align*}
Also, 
\begin{align*}
& H^{\rm{Lie}}_*(\h) \simeq H^{\rm{Lie}}_*(\frak{sl}_n(\br)) \ot
[\Lambda^*(I)]^{\frak{sl}_n(\br)} \\ 
& \simeq  H^{\rm{Lie}}_*(\frak{sl}_n(\br)) \ot
\big\langle \frac{\p}{\p x^1} \w \frac{\p}{\p x^2} \w \, \ldots \, \w 
\frac{\p}{\p x^n} \big\rangle . 
\end{align*}
Thus, $T(K_*) = \sum_{m \geq 0}K_*^{\ot m} = \br$, and the corollary
follows from Theorem \eqref{main-theorem}.  
\end{proof}

\begin{corollary}
Consider $\g = \frak{sl}_2(\bc)$ as a real Lie algebra with real
vector space basis:
\begin{align*}
& v_1 = x_1  \frac{\p}{\p x^1} + x_2 \frac{\p}{\p x^2} 
- x_3 \frac{\p}{\p x^3} - x_4 \frac{\p}{\p x^4}, \\
& v_2 = x_1 \frac{\p}{\p x^2} - x_2 \frac{\p}{\p x^1} 
- x_3 \frac{\p}{\p x^4} + x_4 \frac{\p}{\p x^3} \\
& v_3 = x_1 \frac{\p}{\p x^3} + x_2 \frac{\p}{\p x^4} \\
& v_4 = x_1 \frac{\p}{\p x^4} - x_2 \frac{\p}{\p x^3} \\
& v_5 = x_3 \frac{\p}{\p x^1} + x_4 \frac{\p}{\p x^2} \\
& v_6 = x_3 \frac{\p}{\p x^2} - x_4 \frac{\p}{\p x^1}.
\end{align*} 
For $I = \br^4$ the standard representation of 
$\frak{sl}_2(\bc) \subseteq \frak{sl}_4(\br)$, we have
$$  HL_*(\h) \simeq [\Lambda^*(I)]^{\frak{sl}_2(\bc)}.  $$
\end{corollary}
\begin{proof}
Again, there are no non-trivial $\frak{sl}_2(\bc)$-module
maps from $\frak{sl}_2(\bc)$ to $I^{\w k}$.  Thus, $T(K_*) = \br$ in this 
case as well.  The reader may check that 
$[\Lambda^2(I)]^{\frak{sl}_2(\bc)}$ has a real vector space basis given
by the two elements:  
$$ \frac{\p}{\p x^1} \w \frac{\p}{\p x^3} - \frac{\p}{\p x^2} \w \frac{\p}{\p x^4},
\ \ \ \ \ \frac{\p}{\p x^1} \w \frac{\p}{\p x^4} + 
\frac{\p}{\p x^2} \w \frac{\p}{\p x^3}.  $$
Furthermore, $[\Lambda^4(I)]^{\frak{sl}_2(\bc)}$ is a one-dimensional
(real) vector space on the volume element 
$$ \frac{\p}{\p x^1} \w \frac{\p}{\p x^2} \w \frac{\p}{\p x^3} \w 
\frac{\p}{\p x^4}.  $$
For $k = 1$, 3, we have $[\Lambda^k(I)]^{\frak{sl}_2(\bc)} = 0$.  
\end{proof}

\begin{corollary} \cite{Biyogmam} 
Let $\g = \frak{so}_n(\br)$, $n \geq 3$, and
$$ \al_{ij} = x_i \frac{\p}{\p x^j} - x_j \frac{\p}{\p x^i}
\in \frak{so}_n(\br), \ \ \ 1 \leq i < j \leq n .  $$
For $I = \br^n$ the standard representation of $\frak{so}_n(\br)$, 
we have
$$  HL_*(\h) \simeq [\Lambda^*(I)]^{\frak{so}_n(\br)} \ot T(W),  $$
where $W$ is the one-dimensional vector space with $\h$-invariant
basis element
\begin{align*}
& \om = \sum_{\sigma \in {\rm{Sh}}_{2, \, n-2}} {\rm{sgn}}(\sigma)
\, \al_{\sigma(1) \, \sigma(2)} \ot \epsilon \Big(
\frac{\p}{\p x^{\sigma(3)}} \w \frac{\p}{\p x^{\sigma(4)}} \w \ldots
\w \frac{\p}{\p x^{\sigma(n)}} \Big) \\
& + (-1)^{n+1} \sum_{\sigma \in {\rm{Sh}}_{n-2, \, 2}} {\rm{sgn}}(\sigma)
\, \epsilon \Big(
\frac{\p}{\p x^{\sigma(1)}} \w \frac{\p}{\p x^{\sigma(2)}} \w \ldots
\w \frac{\p}{\p x^{\sigma(n-2)}} \Big) \ot \al_{\sigma(n-1) \, \sigma(n)},
\end{align*} 
and $\epsilon : I^{\w (n-2)} \to I^{\ot (n-2)} \hookrightarrow
\h^{\ot (n-2)}$ is the skew-symmetrization map.  Above, 
${\rm{Sh}}_{p, \, q}$ denotes the set of $p$, $q$ shuffles in
the symmetric group $S_{(p+q)}$.
\end{corollary}
\begin{proof}
There are two non-trivial $\frak{so}_n(\br)$-module maps
$\frak{so}_n(\br) \to I^{\w k}$ to consider
$$  \rho_1 : \frak{so}_n(\br) \to I^{\w 2}, \ \ \ \ \ 
\rho_2 : \frak{so}_n(\br) \to I^{\w (n-2)},  $$
given by
\begin{align*}
& \rho_1 (\al_{ij}) = \frac{\p}{\p x^i} \w \frac{\p}{\p x^j}, \\
& \rho_2 (\al_{ij}) = {\rm{sgn}}(\tau) \,
\frac{\p}{\p x^1} \w \frac{\p}{\p x^2} \w \ldots
\hat{\frac{\p}{\p x^i}} \ldots \hat{\frac{\p}{\p x^j}} \ldots
\w \frac{\p}{\p x^n},
\end{align*}
where $\tau$ is the permutation sending
$$ 1, \ 2, \, \ldots \, ,\, i, \ldots\, , \, j, \, \ldots \, , \, n \ \ \ 
{\rm{to}} \ \ \ i, \ j, \ 1, \ 2, \, \ldots \, , \, n. $$
Now, $\sum_{i<j} \al_{ij} \ot \epsilon (\frac{\p}{\p x^i} \w 
\frac{\p}{\p x^j})$ is not a cycle in the Leibniz complex, while
$$ \lambda = \sum_{\sigma \in {\rm{Sh}}_{2, \, n-2}} {\rm{sgn}}(\sigma)
\, \al_{\sigma(1) \, \sigma(2)} \ot \epsilon \Big(
\frac{\p}{\p x^{\sigma(3)}} \w \frac{\p}{\p x^{\sigma(4)}} \w \ldots
\w \frac{\p}{\p x^{\sigma(n)}} \Big)  $$
is a cycle in $\h^{\ot (n-1)}$, and $\om$ above is a homologous $\h$-invariant
cycle.  For $\sigma \in S_n$, let
\begin{align*}
& a(\sigma) = \sum_{i=1}^{n-2} \al_{\sigma(i) \, \sigma(n-1)} \ot
\al_{\sigma(i) \, \sigma(n)}, \\
& \gamma = \sum_{\sigma \in {\rm{Sh}}_{n-2, \, 2}} {\rm{sgn}}(\sigma)
\, \epsilon \Big(
\frac{\p}{\p x^{\sigma(1)}} \w \frac{\p}{\p x^{\sigma(2)}} \w \ldots
\w \frac{\p}{\p x^{\sigma(n-2)}} \Big) \ot a(\sigma).  
\end{align*}
Then in the Leibniz complex,
\begin{align*}
& d(\gamma) = (n-2) \beta, \\
& \beta = (-1)^{n+1} \sum_{\sigma \in {\rm{Sh}}_{n-2, \, 2}} {\rm{sgn}}(\sigma)
\, \epsilon \Big(
\frac{\p}{\p x^{\sigma(1)}} \w \frac{\p}{\p x^{\sigma(2)}} \w \ldots
\w \frac{\p}{\p x^{\sigma(n-2)}} \Big) \ot \al_{\sigma(n-1) \, \sigma(n)}.
\end{align*}
Thus, $\om$ and $\lambda$ are homologous in $HL_*$.   From 
\cite{Biyogmam},
$$  [\Lambda^*(I)]^{\frak{so}_n(\br)} = 
\big\langle \frac{\p}{\p x^1} \w \frac{\p}{\p x^2} \w \, \ldots \, \w 
\frac{\p}{\p x^n} \big\rangle .  $$
\end{proof}

\begin{corollary}
Let $\g = \frak{so}(3, \, 1)$ and
\begin{align*}
& \al_{ij} = x_i \frac{\p}{\p x^j} - x_j \frac{\p}{\p x^i}
\in \frak{so}(3, \, 1), \ \ \ 1 \leq i < j \leq 3, \\
& \be_{ij} =   x_i \frac{\p}{\p x^j} + x_j \frac{\p}{\p x^i}
\in \frak{so}(3, \, 1), \ \ \ i = 1, \ 2, \ 3, \ \, j = 4.
\end{align*}
For $I = \br^4$ the standard representation of $\frak{so}(3, \, 1)$,
we have
$$  HL_*(\h) \simeq [\Lambda^*(I)]^{\frak{so}(3, \, 1)} \ot T(W),  $$
where $W$ is the one-dimensional vector space with $\h$-invariant
basis element
\begin{align*}
\om =  \, & \al_{12}\ot \Big(\frac{\p}{\p x^3} \w \frac{\p}{\p x^4}\Big)
- \al_{13}\ot \Big(\frac{\p}{\p x^2} \w \frac{\p}{\p x^4}\Big)
+ \al_{23}\ot \Big(\frac{\p}{\p x^1} \w \frac{\p}{\p x^4}\Big) \\
& + \be_{14}\ot \Big(\frac{\p}{\p x^2} \w \frac{\p}{\p x^3}\Big)
- \be_{24}\ot \Big(\frac{\p}{\p x^1} \w \frac{\p}{\p x^3}\Big)
+ \be_{34}\ot \Big(\frac{\p}{\p x^1} \w \frac{\p}{\p x^2}\Big) \\
& - \Big(\frac{\p}{\p x^3} \w \frac{\p}{\p x^4}\Big) \ot \al_{12}
+ \Big(\frac{\p}{\p x^2} \w \frac{\p}{\p x^4}\Big) \ot \al_{13}
- \Big(\frac{\p}{\p x^1} \w \frac{\p}{\p x^4}\Big) \ot \al_{23} \\
& - \Big(\frac{\p}{\p x^2} \w \frac{\p}{\p x^3}\Big) \ot \be_{14}
+ \Big(\frac{\p}{\p x^1} \w \frac{\p}{\p x^3}\Big) \ot \be_{24}
- \Big(\frac{\p}{\p x^1} \w \frac{\p}{\p x^2}\Big) \ot \be_{34}.
\end{align*}
\end{corollary}
\begin{proof}
The proof follows from identifying $\frak{so}(3, \, 1)$ module 
maps $\rho : \frak{so}(3, \, 1) \to I^k$ and constructing
$\h$-invariants via balanced tensors.  Note that the Killing
form to establish $\g \simeq \g^*$ is different for
$\frak{so}_4(\br)$ and $\frak{so}(3, \, 1)$.
\end{proof}

\begin{corollary} \cite{Lodder2}
Let $\g = \frak{sp}_n(\br)$ be the real symplectic Lie algebra with
vector space basis given by the families:
\begin{itemize}
\item[(1)]  $x_k \frac{\p}{\p y^k}$, $\ k = 1, \ 2, \ 3, \ \ldots, \ n$,
\item[(2)]  $y_k \frac{\p}{\p x^k}$, $\ k = 1, \ 2, \ 3, \ \ldots, \ n$,
\item[(3)]  $x_i \frac{\p}{\p y^j} + x_j \frac{\p}{\p y^i}$, $\ 1 \leq i
< j \leq n$,
\item[(4)]  $y_i \frac{\p}{\p x^j} + y_j \frac{\p}{\p x^i}$, $\ 1 \leq i
< j \leq n$,
\item[(5)]  $y_j \frac{\p}{\p y^i} - x_i \frac{\p}{\p x^j}$, $\ i = 1, \
2, \ 3, \ \ldots, \ n$, $\ j = 1, \ 2, \ 3, \ \ldots, \ n$.
\end{itemize}
Let $I = \br^{2n}$ have basis
$$ \frac{\p}{\p x^1}, \ \frac{\p}{\p x^2}, \ \ldots, \ 
\frac{\p}{\p x^n} , \ \frac{\p}{\p y^1}, \ \frac{\p}{\p y^2}, \ \ldots, \ 
\frac{\p}{\p y^n} .  $$
Then
$$  HL_*(\h) \simeq [\Lambda^*(I)]^{\frak{sp}_n} = \Lambda^*(\om_n), $$
where $\om_n = \sum_{i=1}^n \frac{\p}{\p x^i} \w \frac{\p}{\p y^i}$.  
\end{corollary}
\begin{proof}
Since there are no non-trivial $\frak{sp}_n(\br)$-module maps
$\frak{sp}_n(\br) \to I^{\w k}$, we have $T(K_*) = \br$.  The algebra
of symplectic invariants $[\Lambda^*(I)]^{\frak{sp}_n}$ is identified
in another paper \cite{Lodder2}.  
\end{proof}

\end{document}